\documentclass[12pt, reqno]{amsart}
\usepackage[top=3cm, bottom=3.5cm, left=3cm, right=3cm]{geometry}
\usepackage{amsmath,amsfonts,amssymb,amsthm}
\usepackage{mathtools}
\usepackage[utf8]{inputenc}
\usepackage{hyperref}
\usepackage{mathrsfs}
\usepackage{tikz-cd}
\usepackage{lipsum}
\usepackage{csquotes}

\def\Id{\operatorname{Id}}

\usepackage{geometry}
\usepackage{booktabs}
\usepackage{array}
\usepackage{paralist}
\usepackage{verbatim}
\usepackage{subfig}
\usepackage{fancyhdr}
\author{Xi Sisi Shen}
\address{Department of Mathematics\\
  Columbia University\\
  New York, NY 10027, USA}
\email[X. S. Shen]{xss@math.columbia.edu}
\author{Pranay Talla}
\address{Department of Mathematics\\
  Columbia University\\
  New York, NY 10027, YSA}
\email[P. Talla]{pt2555@columbia.edu}

\newtheorem{thm}{Theorem}

\newtheorem{lemma}{Lemma}

\numberwithin{equation}{section}

\begin{document}
\bibliographystyle{amsplain}

\title{Non-preservation of $\alpha$-concavity for the porous medium equation in higher dimensions}

\begin{abstract}
In this short note, we prove that $\alpha$-concavity of the pressure is not preserved for the porous medium equation in dimensions $n=3$ and higher for any $\alpha\in [0,1]\backslash \{\frac{1}{2}\}$. Together with the result of Chau-Weinkove for $n=2$, this fully resolves an open problem posed by V\'asquez on whether pressure concavity is preserved in general for the porous medium equation.
\end{abstract}
\maketitle

\section{Introduction}\label{Intro}

The porous medium equation (PME), also known as the nonlinear heat equation, models the diffusion of an ideal gas in a porous medium. For a gas density function $u(x,t)\ge 0$, where $(x,t)\in \mathbb{R}^n\times (0,\infty)$, the PME is given by
\begin{align}
    \frac{\partial u}{\partial t} = \Delta (u^m), 
    \label{pme}
\end{align}
for a fixed $m>1$ and initial data $u(x,0)=u_0(x)$ for $ x\in\mathbb{R}^n.$ For any $t>0,$ the function sending $x$ to $u(x,t)$ is smooth if we restrict to the set 
\begin{align*}
    \Omega_t = \{x\in\mathbb{R}^n: u(x,t)>0\}. 
\end{align*}
Defining the \textit{pressure} as $v=(m/(m-1))u^{m-1}$, where $u$ is the gas density, we can rewrite the PME in terms of $v$ as
\begin{align}
    \frac{\partial v}{\partial t} = (m-1)v\Delta v + |\nabla v|^2\label{pressure_eqn}
\end{align}
on the set $\Omega_t$ with $v(x,0)=v_0(x)$ for $x\in \Omega_0.$
Given the parabolic nature of the PME, it is of great interest to study whether the concavity of solutions is preserved. For $\alpha>0$ and a function $f$, we say that $f$ is $\alpha$-concave if $f^\alpha$ is concave, in particular, $\alpha=1$ corresponds to the standard definition of concavity. When $\alpha=0$, we say $f^\alpha$ is concave if $\log f$ is concave. When $f$ is twice differentiable, concavity of $f$ is equivalent to $f$ having a negative semi-definite Hessian matrix.
In 2007, V\'asquez posed the following problem in \cite{vasquez07}:\\

\begin{displayquote}
``Prove or disprove the preservation of pressure concavity for the\\ solutions of the PME in several space dimensions." 
\end{displayquote}
\vspace{4mm}

For $n=1$, B\'enilan and V\'azquez \cite{benilan-vasquez87} have shown that concavity of pressure is preserved. Daskopoulos-Hamilton-Lee \cite{dhl01} prove that root concavity of pressure, which corresponds to $\alpha=\frac{1}{2}$, is preserved in all dimensions. 

On the other hand, Ishige-Salani \cite{ishige-salani10} show that there exists an $\alpha\in (0,\frac{1}{2})$ for which $\alpha$-concavity of the pressure is not preserved. Extending this result, Chau-Weinkove \cite{chau-weinkove24} show that $\alpha$-concavity of the pressure for $\alpha\in[0,1]\backslash\{\frac{1}{2}\}$, and particularly in the case of $\alpha=1$, is not preserved in general for $n=2$. They show this by constructing explicit examples where $\alpha$-concavity is instantaneously lost for any $\alpha\in [0,1]\backslash\{\frac{1}{2}\}.$ Chau-Weinkove also prove that $\alpha$-concavity of solutions to the one-phase Stefan problem is not preserved for $\alpha\in[0,\frac{1}{2})$ in \cite{chau-weinkove21}.

In this paper, we extend the result of \cite{chau-weinkove24} to higher dimensions, showing that $\alpha$-concavity of the pressure for $\alpha\in [0,1]\backslash\{\frac{1}{2}\}$ is not preserved in general for the porous medium equation in any dimension $n$, and in particular for $\alpha=1$. Due to the result in \cite{dhl01} for $\alpha=\frac{1}{2}$, the above result for $\alpha\in [0,1]\backslash\{\frac{1}{2}\}$ is sharp. Specifically, we establish the following theorem:
\begin{thm}\label{main_thm}
    Let B be the open unit ball in $\mathbb{R}^n$ centered at the origin. For all $n$, given $\alpha \in [0,1] \backslash \{\frac{1}{2}\}$, there exists $v_0\in C^{\infty}(\overline{B})$ which is strictly positive on B and vanished on $\partial B$ with the following properties:
    \begin{enumerate}
        \item $v_0$ is concave on B.
        \item $\nabla v_0$ does not vanish at any point of $\partial B$.
        \item Let $v(t)$ be the solution of the porous medium equation \eqref{pressure_eqn} starting at $v_0$. Then there exists $\delta>0$ such that $v(t)$ is not $\alpha$-concave in a neighborhood of the origin for $t\in (0,\delta)$.
    \end{enumerate}
\end{thm}

 The paper is organized as follows. In Section \ref{non-pres}, we construct a specific concave function in Lemma \ref{construction} and use it to prove Theorem \ref{main_thm} by providing an explicit example for each $\alpha\in[0,1]\backslash\{\frac{1}{2}\}$ for which $\alpha$-concavity is instantaneously lost at an interior point of the domain. 

\section{Non-preservation of interior concavity}\label{non-pres}
Let $v$ be a local positive smooth solution to the PME in terms of the pressure \eqref{pressure_eqn} and let $w=v^\alpha$. We first compute the closed form expression for $\frac{\partial}{\partial t} w_{11}$ for $\alpha> 0$  (see Equation (3.2) of \cite{chau-weinkove24}), where the subscripts on $w$ refer to spatial partial derivatives and indices are assumed to be summed over:

\begin{align}\label{time_derivative}
    \frac{\partial}{\partial t} w_{11} =& (m-1)w^{\frac{1}{\alpha}}w_{kk11} + \frac{2(m-1)}{\alpha}w^{\frac{1}{\alpha}-1}w_1w_{kk1} \nonumber\\
    &+ \frac{(m-1)}{\alpha}\Big(\frac{1}{\alpha}-1\Big)w^{\frac{1}{\alpha}-2}w_1^2w_{kk}+\frac{(m-1)}{\alpha}w^{\frac{1}{\alpha}-1}w_{11}w_{kk} \nonumber\\
    &+ \frac{1}{\alpha}(1+(m-1)(1-\alpha))\Big\{\big(\frac{1}{\alpha}-2\big) \big(\frac{1}{\alpha}-1\big) w^{\frac{1}{\alpha}-3}w_1^2w_k^2 \\
    &+ \Big(\frac{1}{\alpha}-1\Big)w^{\frac{1}{\alpha}-2}w_{11}w_k^2 + 4\big(\frac{1}{\alpha}-1\big)w^{\frac{1}{\alpha}-2}w_1w_kw_{k1} \nonumber\\
    &+ 2w^{\frac{1}{\alpha}-1}w_{1k}^2+2w^{\frac{1}{\alpha}-1}w_kw_{k11}\Big\}\nonumber.
\end{align}

We must also compute $\frac{\partial}{\partial t}w_{11}$ for $\alpha=0$, as follows:

\begin{align}\label{time_derivative_alpha_0}
    \frac{\partial}{\partial t}w_{11} =& (m-1) e^w w_{kk11} + 2(m-1)e^w w_1w_{kk1}\nonumber \\
    &+ (m-1)e^w w_1^2 w_{kk} + (m-1)e^w w_{11} w_{kk} + me^w w_1^2 w_k^2 \\
    &+ me^w w_1 w_k w_{k1} + 2me^w w_{1k}^2 + 2me^w w_k w_{k11}.\nonumber
\end{align}

Let $B_\rho(0)$ be the open ball of radius $\rho$ in
$\mathbb{R}^n$ centered at the origin. We now construct a specific concave function on $\overline{B_\rho(0)}$ in the following lemma:
\begin{lemma}\label{construction}
    For each $\alpha\in [0,1]\backslash\{\frac{1}{2}\}$ there exists a $\rho>0$ and a smooth positive function $w$ on $\overline{B_\rho(0)}$ such that
    \begin{enumerate}
        \item $w_{ii}(0)<0$ for $2\le i\le n$ and $w_{ij}(0)=0$ otherwise.
        \item $(D^2 w)<0$ on $\overline{B_\rho(0)}\backslash \{0\}.$
        \item If $\alpha\neq 0$ then the right hand side of \eqref{time_derivative} is positive at the origin. If $\alpha = 0$ then the right hand side of \eqref{time_derivative_alpha_0} is positive at the origin.
    \end{enumerate}
\end{lemma}
\begin{proof} We will consider two separate cases:\\

    \textbf{Case 1.} $\alpha\in [0,\frac{1}{2})$
 or $\alpha=1$. For $a>0$ a constant to be determined, define 
 \begin{align*}
    w = 1 + ax_1 - x_1^4 + \Big(\sum_{i=2}^n x_1x_i^2 -x_i^2 - 2x_1^2x_i^2\Big).
 \end{align*}
 Then the only nonzero components of the Hessian matrix $(D^2 w)$ are
 \begin{align*}
     w_{11}&=-12x_1^2-\sum_{i=2}^n 4x_i^2,
\end{align*}
and for $i=2,\ldots n$,
\begin{align*}
     w_{1i}&=w_{i1}=2x_i-8x_1x_i\\
     w_{ii}&=2x_1-2-4x_1^2.
 \end{align*}
 Condition (1) follows immediately by substituting $x_i=0$ for all $i=1,\ldots,n$. To verify Condition (2), notice that the determinant of the $j^{th}$ leading principal minor, for $j=2,\ldots, n$ is given by
 \begin{align*}
     \det([D^2w]_j)=(-1)^j2^j\Big(6x_1^2+2\sum_{i=2}^nx_i^2-\sum_{i=2}^jx_i^2\Big) + \text{higher order terms}.
 \end{align*}
 Note that $$6x_1^2+2\sum_{i=2}^nx_i^2-\sum_{i=2}^jx_i^2>0$$ away from the origin. Using the leading principal minors method, and the fact that we can choose $\rho>0$ sufficiently small so that $$ \det([D^2w]_1)<0, \ \det([D^2w]_{(2k)})>0 , \ \det([D^2w]_{(2k+1)})<0 $$ for $(2k),(2k+1)\in \{2,\ldots,n\}$, we have that the Hessian $(D^2w)$ is negative definite. Finally, let us check Condition (3). Observe that at the origin, $w=1$ and the only non-negative derivatives of $w$ are
 \begin{align*}
     w_1&=a\\
     w_{kk}&=-2\\
     w_{kk1}&=2\\
     w_{1111}&=-24\\
     w_{kk11}&=-8
 \end{align*}
 for $k=2,\ldots,n$. For $\alpha> 0$, the right hand side of \eqref{time_derivative} is positive at the origin exactly when
 \begin{align*}
     &-(24+8(n-1))(m-1)+\frac{4(m-1)}{\alpha}a(n-1)-\frac{2(m-1)}{\alpha}\Big(\frac{1}{\alpha}-1\Big)a^2(n-1)\\
     & \hspace{26mm} +\frac{1}{\alpha}(1+(m-1)(1-\alpha))\Big(\frac{1}{\alpha}-1\Big)\Big(\frac{1}{\alpha}-2\Big)a^4>0.
 \end{align*}
\vspace{4mm}
If $\alpha=1$ then the only non-zero terms are the first two and the second term will dominate by choosing $a$ large enough.\\

\noindent When $\alpha\in(0,1/2)$, we see that the last term is positive and will dominate the other three terms for sufficiently large $a$.\\

\noindent If $\alpha=0$, the right hand side of \eqref{time_derivative_alpha_0} is positive at the origin precisely when 
 \begin{align*}
     -(24+8(n-1))(m-1)+4(m-1)(n-1)a -2(m-1)(n-1)a^2 + ma^4>0
 \end{align*}
 which again can be made positive by letting $a$ sufficiently large. This confirms that Condition (3) holds for all $\alpha\in[0,1/2)$ or $\alpha=1$.\\

    \textbf{Case 2.} $\alpha\in (\frac{1}{2},1)$. For $b>0$ a constant to be determined, define 
 \begin{align*}
    w = 1 + \frac{\alpha (\frac{3}{2}-\alpha)^{\frac{1}{2}}}{b(1-\alpha)}x_1 -\frac{x_1^4}{12b^2} + \sum_{i=2}^n\Big(-b^2x_i^2 + b\bigg(\frac{3}{2}-\alpha\bigg)^{\frac{1}{2}}x_1x_i^2-x_1^2x_i^2 \Big).
 \end{align*}
 Then the only nonzero components of the Hessian matrix $(D^2 w)$ are
 \begin{align*}
     w_{11}&=-\frac{x_1^2}{b^2} - 2\sum_{i=2}^n x_i^2,
\end{align*}
and for $i=2,\ldots n$,
\begin{align*}
     w_{1i}&=w_{i1}=2b\bigg(\frac{3}{2}-\alpha\bigg)^{\frac{1}{2}}x_i-4x_1x_i\\
     w_{ii}&=-2b^2+2b\bigg(\frac{3}{2}-\alpha\bigg)^{\frac{1}{2}}x_1-2x_1^2.
 \end{align*}
 Condition (1) follows immediately by substituting $x_i=0$ for all $i=1,\ldots,n$, as $w_{ii} < 0$ for all $i=2,\ldots, n$ and all other components of the Hessian will become zero. To verify Condition (2), notice that the determinant of the $j^{th}$ leading principal minor, for $j=2,\ldots, n$ is given by
 \begin{comment}\begin{align*}
     \det([D^2w]_j)&= (-1)^{n+1}w_{1n}\Big((-1)^nw_{1n}w_{22}^{n-2}\Big)+w_{nn}D^2_{n-1}w \\
     &= -w_{22}\Big(\sum_{i=2}^nw_{1i}^2\Big)+w_{22}w_{11}
 \end{align*}
This simplifies to a closed form of
\end{comment}
 \begin{align*}
     \det([D^2w]_j)&=(-1)^j(2b^2)^{j-1}\Big(\frac{x_1^2}{b^2} +2\sum_{i=2}^nx_i^2 - 2(3/2-\alpha)\sum_{i=2}^jx_i^2\Big) + \text{higher order terms}.
 \end{align*}
Since $\alpha\in(\frac{1}{2},1)$, $$\frac{x_1^2}{b^2} +2\sum_{i=2}^nx_i^2 - 2(3/2-\alpha)\sum_{i=2}^jx_i^2>0$$ away from the origin for all $j=2,\ldots, n$. It follows that for $\rho>0$ sufficiently small, $$ \det([D^2w]_1)<0, \ \det([D^2w]_{(2k)})>0 , \ \det([D^2w]_{(2k+1)})<0 $$ for $(2k),(2k+1)\in \{2,\ldots,n\}$. By the leading principal minors method, we have that the Hessian $(D^2w)$ is negative definite. To confirm Condition (3) holds, note that $w=1$ at the origin and the only non-vanishing derivatives of $w$ are, for $k=2,\ldots,n$,
\begin{align*}
    w_1&= \frac{\alpha(\frac{3}{2}-\alpha)^{\frac{1}{2}}}{b(1-\alpha)}\\
    w_{kk} &= -2b^2\\
    w_{kk1} &= 2b\Big(\frac{3}{2}-\alpha\Big)^{\frac{1}{2}} \\
    w_{1111} &= -\frac{2}{b^2}\\
    w_{kk11} &= -4.
\end{align*}

Substituting these values into \eqref{time_derivative}, we see that

\begin{align*}
    \frac{\partial}{\partial t}w_{11} =& (m-1)\Big(-\frac{2}{b^2}-4(n-1)\Big) + \frac{2(m-1)}{\alpha}\frac{\alpha(3/2-\alpha)}{(1-\alpha)}2(n-1) \\
    &+ \Big(\frac{m-1}{\alpha}\Big)\Big(\frac{1}{\alpha}-1\Big)\Big(\frac{\alpha^2(3/2-\alpha)}{(1-\alpha)^2}\Big)(-2(n-1))\\
    &+\frac{1}{\alpha}(1+(m-1)(1-\alpha))\Big(\frac{1}{\alpha}-2\Big)\Big(\frac{1}{\alpha}-1\Big)\Big(\frac{\alpha^2(3/2-\alpha)}{b^2(1-\alpha)^2}\Big)^2,
\end{align*}

which simplifies to  

\begin{align*}
    \frac{\partial}{\partial t}w_{11} =& (m-1)\Big(-\frac{2}{b^2}-4(n-1)\Big) + 2(m-1)\frac{(3/2-\alpha)}{(1-\alpha)}(n-1)\\
    &+ \frac{1}{\alpha}(1+(m-1)(1-\alpha))\Big(\frac{1}{\alpha}-2\Big)\Big(\frac{1}{\alpha}-1\Big)\Big(\frac{\alpha^2(3/2-\alpha)}{b^2(1-\alpha)^2}\Big)^2 \\ 
    =& \frac{-2(m-1)}{b^2}-\frac{C_{\alpha,m}}{b^4}+ (m-1)(n-1)\Big(\frac{2\alpha-1}{1-\alpha}\Big).
\end{align*}

Here, $C_{\alpha,m}$ is a positive constant depending only on $m$ and $\alpha$.
Since $\frac{1}{2} < \alpha < 1$, the third term in the last line is strictly positive and thus we can take $b$ sufficiently large to ensure that Condition (3) of Lemma \ref{construction} is satisfied.

\end{proof}
In order to prove the Main Theorem, we follow a similar strategy to the one used to prove Theorem 1.1 in \cite{chau-weinkove24}. It suffices to construct a function $v$ for each $\alpha \in [0,1] \backslash \{\frac{1}{2}\}$ such that at time $t=0$, $v_0$ is $\alpha$-concave with non-vanishing gradient at the boundary of the domain $B$. The constructed $v$ would then need to lose $\alpha$-concavity instantaneously after $t=0$. This can be achieved by constructing $v$ such that it has exactly one second derivative (here, we use $v_{11}$) that is equal to zero at $t=0$ and positive for $t \in (0,\delta)$.

Because it is easier to verify the condition of concavity than the condition of $\alpha$-concavity, we will indirectly construct $v$ by instead constructing $w=v^{\alpha}$. We then rewrite the three conditions of Theorem \ref{main_thm} as in Lemma \ref{construction} above. Specifically, for any function $w$ that satisfies conditions (1) and (3) in Lemma \ref{construction}, the corresponding $v$ will satisfy condition (3) of Theorem \ref{main_thm}; similarly, any $w$ that satisfies (2) of Lemma \ref{construction} will have a corresponding $v$ satisfying (1) and (2) of Theorem \ref{main_thm}.

\begin{proof}[Proof of Theorem \ref{main_thm}]
We have constructed a positive function $w$ on $\overline{B_\rho(0)}$ for some $\rho>0$ that is $\alpha$-concave everywhere at $t=0$ and not $\alpha$-concave immediately after $t=0$. We will use $w$ to construct a function $v_0$ which in addition has non-vanishing gradient on $\partial B_\rho(0)$. The result in Theorem \ref{main_thm} then follows by scaling to the unit ball. To proceed, we first use $w$ to construct a function $\tilde{w}$ that will have non-vanishing gradient on the boundary. This construction will build off of auxiliary functions defined in the proof of Theorem 1.1 in  \cite{chau-weinkove24}. Define a smooth cutoff function $\psi: \overline{B_\rho(0)} \rightarrow [0,1]$ such that 

\begin{align*}
    \psi(x) = \begin{cases} 
                    1 & x\in \overline{B_\frac{\rho}{2}(0)} \\
                    0 & x\in \overline{B_{\rho}(0)}\backslash B_{\frac{3\rho}{4}(0)}
                \end{cases}
\end{align*}
whose first and second derivatives are each bounded by some constant. Since $\psi$ is a fixed function, there exists a uniform constant $C$ such that 
\begin{align}\label{psi_bound}
    |D\psi| + |D^2\psi| \leq C.
\end{align}
Furthermore, define a radial concave function $F: \overline{B_{\rho}(0)} \rightarrow \mathbb{R}$ as $F(x) = f(|x|)$, for some continuous decreasing concave function $f: [0,\rho) \rightarrow \mathbb{R}$.

For the case of $\alpha \neq 0,1$, 
\begin{align*}
    f(r) := \begin{cases} 
                    \big(1+\frac{\alpha}{4}\big)\big(\frac{\rho}{2}\big)^{\alpha} &  0\leq r\leq \frac{\rho}{4} \\
                    (\rho-r)^{\alpha} & \frac{\rho}{2}\leq r \leq \rho
                \end{cases}
\end{align*}
$F$ is therefore smooth on $B_{\rho}(0)$ and vanishes on $\partial B_{\rho}(0)$. Crucially, the derivatives of $F$ are bounded such that they are zero inside $B_{\frac{\rho}{4}}(0)$ and such that the following three conditions hold:
\begin{align*}
    f' &\leq -\alpha\big(\frac{\rho}{2}\big)^{\alpha-1} \\
    f' &\leq -\alpha\big(\frac{\rho}{2}\big)^{\alpha-1} \\
    f'' &\leq \alpha(\alpha-1)\big(\frac{\rho}{2}\big)^{\alpha-2}. \\
\end{align*}
We write this as $$f',f'' \leq -\frac{1}{C}$$ for a uniform positive constant $C$, so that on $B_{\rho}(0) \backslash B_{\frac{\rho}{2}}(0)$,

\begin{align}\label{D2Fbound}
    D^2F \leq -\frac{1}{C}\Id.
\end{align}
We now define $\tilde{w}: \overline{B_{\rho}(0)} \rightarrow \mathbb{R}$ as $\tilde{w} = AF + \psi w$ for a positive constant $A$. Because $f$ is constant and because $\psi$ is 1 within $B_{\frac{\rho}{4}}(0)$, the three conditions of Lemma 1 are satisfied by $\tilde{w}$ given our construction of $w$. Furthermore, because $\psi$ is a cutoff function equal to 0 near $\partial B_\rho(0)$, the gradient of $\tilde{w}$ on $\partial B_\rho(0)$ is equal to the gradient of $A^{\frac{1}{\alpha}}(\rho-|x|))$ which is non-vanishing. Finally, we also note that the construction of $F$ ensures $D^2\tilde{w} <0$ outside of the origin. To see this, we consider the two cases of $0\leq r\leq\frac{\rho}{2}$ and $\frac{\rho}{2}< r \leq \rho$. 

In the first case, we can simplify 

\begin{align}\label{D2wfull}
    D^2\tilde{w} = AD^2F+\psi D^2w + (D^2\psi)w + 2D\psi \cdot Dw
\end{align}
by substituting $D\psi = 0$ and \eqref{D2Fbound} to compute that

$$D^2\tilde{w} = AD^2F + D^2w < 0.$$
This is negative because both $F$ and $w$ are constructed to be concave. 

In the second case, we simplify \eqref{D2wfull} through substituting $\psi = 0$ and the derivative bounds on $\psi$ as stated in \eqref{psi_bound} to arrive at the following equation, with $A$ chosen suitably large.

$$D^2\tilde{w} = \Big(C-\frac{A}{C}\Big)\Id < 0.$$

Note that this inequality is only satisfied with a choice of a large constant $A$ that depends on the values of $\alpha$ and $\rho$. As a result, we have shown that $\tilde{w}$ is concave on $B_{\rho}(0)$ and has $D^2\tilde{w} < 0$ everywhere in $B_{\rho}(0)$ except for the origin. This is therefore suitable to form the foundation of our constructed function $v$ that will lose $\alpha$-concavity after $t=0$. Specifically, we can define the initial data of $v$ as follows:
$$v_0 = \tilde{w}^{\frac{1}{\alpha}}.$$
Near $\partial B_{\rho}(0)$, we have that $\tilde{w} = AF \Rightarrow v_0 = A^{\frac{1}{\alpha}}(\rho - |x|)$, meaning that $v_0$ is smooth. Furthermore, the solution $v$ such that $v(0)=v_0$ is smooth, as a result of Proposition 7.2 in \cite{vasquez07}. Furthermore, the gradient of $v$ does not vanish at $\partial B_{\rho}(0)$ as the term $AF$ has nonzero gradient at that boundary.\\

Let $\lambda_1(x,t)$ be the largest eigenvalue of the Hessian matrix $D^2 v^\alpha(x,t)$ at any given time $t$ and coordinate $x$. We know that $\lambda_1(x,t)$ is distinct from all other eigenvalues at the origin, since $\lambda_1(0,0) = 0$ and all the other eigenvalues are negative at the origin at $t=0$. Hence, $\lambda_1$ is smooth around a small neighborhood around the origin. The fact that $w=v^\alpha$ loses $\alpha$-concavity immediately after $t=0$ then follows form the fact that

$$\frac{\partial}{\partial t} \lambda_1 (0) = \frac{\partial}{ \partial t}(v^\alpha)_{11}\Big|_{x=0} > 0.$$
As a result,
$$\lambda_1(x,t)>0 \text{ for all } (x,t)\neq (0,0).$$
The continuous function $\lambda_1$ is zero and increasing at the origin; in other words, there exists an $\epsilon > 0$ such that $\lambda_1 > 0$ for $t\in (0,\epsilon)$. This means that this construction of $w$ satisfies the three properties of Theorem \ref{main_thm}.

For the case of $\alpha=1$, 
\begin{align*}
    f(r) := \begin{cases} 
                    \frac{7\rho^2}{8} &  0\leq r\leq \frac{\rho}{4} \\
                    \rho^2-r^2 & \frac{\rho}{2}\leq r \leq \rho.
                \end{cases}
\end{align*}

We can define the radially symmetric function $F$ and the function $\tilde{w}$ exactly as above, with the alternative initial data $v_0 = \tilde{w}$ to see that the resulting construction of $w$ satisfies Theorem \ref{main_thm}. 

Finally, for the case of $\alpha=0$, 
\begin{align*}
    f(r) := \begin{cases} 
                    \log{(\frac{\rho}{2})} &  0\leq r\leq \frac{\rho}{4} \\
                    \log{(\rho -r)} & \frac{\rho}{2}\leq r \leq \rho.
                \end{cases}
\end{align*}
For this construction, we can instead take $A,\psi =1$, resulting in $D^2\tilde{w} < 0$ on the entire punctured ball $B_{\rho}(0)\backslash \{0\}$. We can then define initial data $v_0 = e^{\tilde{w}}$.  This initial data is log concave (corresponding to the notion of $\alpha$-concavity in the case where $\alpha=0$), smooth on the closure of the ball, and zero with non-vanishing gradient at the boundary. As above, the resulting construction of $w$ satisfies the three properties of Theorem \ref{main_thm}. This completes the proof.

\end{proof}

\bibliography{references}
\end{document}